\documentclass[11pt, a4paper]{article}
\usepackage{graphicx}
\usepackage{mathrsfs}
\usepackage{CJK}
\usepackage{bbm}
\usepackage{stmaryrd}
\usepackage{amsmath}
\usepackage{cases}
\usepackage{amsfonts}
\usepackage{latexsym,amssymb,amsmath,mathrsfs,amsbsy, amsthm}
\usepackage[usenames]{color}
\usepackage{xspace,colortbl}
\usepackage[dvipdfm,
pdfstartview=FitH, CJKbookmarks=true, bookmarksnumbered=true,
bookmarksopen=true,
citecolor=blue 
,colorlinks=true
,pdfborder=1 
,pdfstartpage=2 
,pdfstartview=Fit]{hyperref} 
\allowdisplaybreaks

\newtheorem{thm}{Theorem}[section]
\newtheorem{lem}[thm]{Lemma}

\newtheorem{rem}{Remark}
\newtheorem{coro}[thm]{Corollary}

\newtheorem{defnm}{Definition}[section]
\newtheorem{prop}[thm]{Proposition}
\newcommand{\nn}{\nonumber}

\numberwithin{equation}{section}
\begin{document}
\title{\bf\Large Harmonic Maps with Potential from $\mathbb{R}^2$ into $S^2$}

\author{Ruiqi JIANG\thanks{Supported by NSFC, Grant No. 10990013}}

\date{}
\maketitle

\begin{abstract}
We study the existence problem of harmonic maps with potential from $\mathbb{R}^2$ into $S^2$.
For a specific class of potential functions on $S^2$, we give the sufficient and necessary conditions for  the existence of equivariant solutions of this problem. As an application, we generalize and improve the results
on the Landau-Lifshitz equation from $\mathbb{R}^2$ into $S^2$ in \cite{G_S} due to Gustafson and Shatah.
\end{abstract}


\section{Introduction}

Let $(M,g)$ and $(N,h)$ be two Riemannian manifolds. As is well known, a map $u_0: M \to N$ is called a harmonic map iff it
is critical with respect to the energy functional $E(u)$. See \cite{E_S} for the precise definitions. The notion of harmonic maps with potential is first suggested
by Ratto (\cite{R}). Given a ``potential" function $H: N \to \mathbb{R}$, a map $u_0: M \to N$ is called a {\it harmonic map with potential} $H$ iff it is critical with respect to
the functional
$$  F(u) \equiv E(u) + \int_M H(u) dV_g. $$

In this paper, we consider the existence problem of harmonic maps with potential in the special case where $(M,g)$ is the Euclidean 2-plane $\mathbb{R}^2$, and $(N,h)$ is the unit 2-sphere $S^2$ in $\mathbb{R}^3$, i.e.
$$ S^2 = \{ x\in \mathbb{R}^3: \hskip .2cm |x|^2 = 1\}.$$
Then, if we set $u = (u_1, u_2, u_3)\in \mathbb{R}^3$, the energy is simply
$$ E(u) = {1\over 2}\int_{\mathbb{R}^2} |\nabla u|^2 dx,$$
where
$$  |\nabla u|^2 =  \sum_{i=1}^3 |\nabla u_i|^2.$$
Also, we will assume $H(u) = G(d(u))$ for a function $G:[0,\pi]\to \mathbb{R}$, where $d(u)$ denotes the geodesic distance from $u\in S^2$ to the north pole $P=(0,0, 1)$.
This assumption on $H$ enables us to seek for solutions which are equivariant with respect to the $S^1= O(2)$ actions on both the domain $\mathbb{R}^2$ and the target
$S^2$. If we identify the $(x_1, x_2)$-plane with the complex plane ${\bf C}$, and consider $\mathbb{R}^3 = {\bf C} \oplus \mathbb{R}^1$, then an $m$-equivariant map $u$ takes
the following form.
\begin{eqnarray}\label{(1)}
  u(r, \theta) = \sin h(r) e^{i m\theta} + \cos h(r)\cdot e_3
\end{eqnarray}
where $h: [0, \infty) \to \mathbb{R}^1$ with $h(0)=0$, $m$ is an non-zero integer, and $e_3=(0,0, 1)$ is the unit vector.
For such $m$-equivariant maps, the energy $F$ reduces to a functional on the function $h$ as follows. (We omit the factor $2\pi$ in the integrals.)
\begin{eqnarray}\label{(2)}
J(h) = {1\over 2} \int_0^\infty \left( h'^2 + {{m^2\sin^2 h}\over {r^2}}\right) r dr + \int_0^\infty G(h)\, r dr.
\end{eqnarray}
Moreover, if $h$ is a critical point of $J$, then the map $u$ defined in (\ref{(1)}) is a critical point of $F$, hence a harmonic map with potential for the chosen potential function $H$. Thus, the question of  finding harmonic maps with potential  is reduced to solving the following O.D.E., which is the Euler-Lagrange equation of $J(u)$, with suitable boundary conditions at $r=0$ and $r=\infty$.
\begin{eqnarray}\label{(3)}
h'' + {1\over r} h' - {{m^2 \sin h \cos h }\over {r^2}}
= g(h)
\end{eqnarray}
where $g(h) = G'(h)$. The boundary conditions we assume will be
\begin{eqnarray}\label{(4)}
 h(0)=0, \,\,\,\, h(\infty)=\pi
\end{eqnarray}
Such a problem has been considered by some authors in connection with Landau-Lifshitz type equations.  Gustafson and Shatah (\cite{G_S}), when looking for periodic solutions to certain  Landau-Lifshitz type equation, studied the above problem with $g(h) = \lambda \sin h \cos h + \omega \sin h$. Their result shows that, if
$\lambda>0$, and $\omega>0$ is small, then the problem has a solution which has finite energy and increases monotonely from $0$ to $\pi$. In \cite{H_L}, Hang and Lin considered the same equation with $g(h) = -\lambda \sin h\cos h$, but the condition at infinity is replaced by $h(\infty) = \pi/2$. They prove for each $\lambda>0$ there exists a unique solution of infinite energy for their problem.

In this paper, we will consider a class of functions $g$ in (\ref{(3)}) for which the solvability question of problem  (\ref{(3)})-(\ref{(4)}) can be completely answered. The functions $g$ in this class satisfy the following conditions.

$(i)$ There exists $\xi \in (0,\pi)$ such  that
$$\left\{ \begin{array}{l}
 g(0)=g(\xi)=g(\pi)=0, \\
 g(x) > 0,\quad x \in (0,\xi), \\
 g(x) < 0, \quad x \in (\xi,\pi); \\
 \end{array} \right.$$

$(ii)$ $\quad \int_0^{\pi}g(x)dx>0$;\\

$(iii)$  $\quad g'(\pi)>0. $

We choose the potential function $G$ in  (\ref{(2)}), which is a primitive of $g$, to be
\begin{eqnarray}\label{(5)}
G(x) = -\int_x^\pi g(t) dt,
\end{eqnarray}
so that $G(\pi) =0$ and $G(0) < 0$.

Notice that, the function $g$ in \cite{G_S} falls into our class. Hence, our result is a generalization and improvement of theirs.

It is well known that, when $g\equiv 0$ in (\ref{(3)}), there is a family of solutions $\varphi_\lambda$ to (\ref{(3)})-(\ref{(4)})
which corresponds to a family of harmonic maps from $\mathbb{R}^2$ onto $S^2$ of degree $m>0$. These solutions have the following explicit expression.
$$ \varphi_\lambda(r) = 2 \arctan [ (\lambda r)^{m}].\quad \lambda >0$$
Now we can state our main result.

\begin{thm}\label{thm:1} Assume that the function $g\in C^\infty([0,\, \pi])$ satisfies
$(i) - (iii)$, $m\neq 0$ is an integer and that $G$ is as in
(\ref{(5)}). The problem (\ref{(3)})-(\ref{(4)}) has  solutions with
$0< h(r) < \pi$ on $(0, \infty)$ if and only if we have
$$  0< \int_0^\infty G(\varphi_1(r)) \, r dr \leq \infty.$$
Moreover, the solutions we obtain satisfy $h'(r)>0$ on $(0, \infty)$ and converge to $\pi$ exponentially as $r\to \infty$.
\end{thm}

\begin{rem}\label{rem:1}
Since $$\int_0^\infty G(\varphi_\lambda(r)) \, r dr={1
\over \lambda^2}\int_0^\infty G(\varphi_1(r)) \, r dr,$$ we can
replace $\varphi_1(r)$ by $\varphi_{\lambda}(r)$ for $\lambda >0$.
\end{rem}

\begin{rem}\label{rem:2}
In fact, in Theorem \ref{thm:1} and throughout the paper, we only need to assume that $g \in C^{\alpha(m)}([0,\pi])$, where $\alpha(m)=max\{1,|m|-2\}$, if $m\ne0$ is fixed.

\end{rem}

Our method for the proof is basically a combination of the shooting method for O.D.E.'s,  the variational method for obtaining solutions to certain
boundary value problems and the blow-up analysis for determining the behavior of solutions with large initial data. We will also repeatedly use a Pohozaev type identity in our analysis. (The name ``Pohozaev" usually means such an identity can be
obtained by a domain variation along a conformal vector field, and in our case the vector field is $r {\partial \over {\partial r}}$.)

In the next section we would consider an initial value problem of
O.D.E. (\ref{(3)}) with the singularity at $r=0$ and prove its existence,
uniqueness and continuous dependence on initial data. In Section 3, by
qualitative analysis, we will establish a series of lemmas to
characterize the behavior of solutions of O.D.E. (\ref{(3)}) under suitable
assumptions. In Section
4 we discuss the existence of the boundary value problems of
O.D.E. (\ref{(3)}) by variational methods. In Section 5, we give the proof of
Theorem (\ref{thm:1}) by shooting method. Finally, we apply our result to certain Landau-Lifshitz type equations in section 6.\\

\noindent \textbf{Convention:} For convenience, we always assume that $m>0$ without further comment.

\section{The existence of solutions to the initial value problems}
In this section, we need to consider the initial value problem of
O.D.E. (\ref{(3)}) with the singularity at $r=0$ and want to prove its
existence, uniqueness and the continuous dependence on the initial data.

For the convenience, we rewrite (\ref{(3)}) as following form:
$$h'' + \frac{1}{r}h' - \frac{{m^2 }}{{r^2 }}\sin h \cos h  - g(h) =
0.$$
We consider the following initial value problem:
\begin{eqnarray}
&&h'' + \frac{1}{r}h' - \frac{{m^2 }}{{r^2 }}\sin h \cos h  - g(h) =0, \quad r \in (0,+\infty ) \label{eqn:2.1}\\
&&h(0) = 0,\quad h^{(m)}(0) = m!a \label{eqn:2.2},
\end{eqnarray}
where $g(x) \in C^\infty(\mathbb{R})$, $ \|g\|_{C^1} \le C<\infty$, $a \in
\mathbb{R}$ and $h^{(m)}$ denotes the $m$-order derivative of $h$.

\begin{defnm}
If $h(r) \in C^m [0,+\infty) \cap C^\infty (0,+\infty)$ satisfies
(\ref{eqn:2.1})-(\ref{eqn:2.2}), then $h(r)$ is called a solution to
(\ref{eqn:2.1})-(\ref{eqn:2.2}).
\end{defnm}

\begin{rem}
If $h(r)$ is a solution of (\ref{eqn:2.1})-(\ref{eqn:2.2}), by
substituting the asymptotic expansion of $h(r)$ at $r=0$ to the
equation (\ref{eqn:2.1}), we see that k-th derivative of $h(r)$
evaluated at the point 0 with $k \leqslant m-1$ is zero, i.e.
$$h^{(k)} (0) = 0,\quad 0 \leqslant k \leqslant m - 1.$$ So the
initial value (\ref{eqn:2.2}) is given reasonably for (\ref{eqn:2.1}).
\end{rem}

We shall adopt the contraction map principle, whic is different from that in \cite {HJ.F, H_L}, to address the problem
of existence and uniqueness of local solutions to
(\ref{eqn:2.1})-(\ref{eqn:2.2}). Then, by means of the standard
existence and uniqueness theory on ordinary differential equation,
we can extend the local solution to the whole interval $[0,+\infty)$.

It is easy to see that (\ref{eqn:2.1}) can be expressed in the
following form:
\begin{eqnarray}\label{eqn:2.3}
(rh')' = \frac{{m^2 }}
{r}\sin h \cos h  + g(h)r
\end{eqnarray}
Hence, (\ref{eqn:2.3}) is equivalent to the integral equation:
\begin{eqnarray}\label{eqn:2.4}
h(r) = \int_0^r {\frac{1}{s}} \int_0^s {(\frac{{m^2 }}
{{2t}}} \sin 2h + g(h)t)dtds.
\end{eqnarray}
Let
\begin{eqnarray}\label{eqn:2.5}
h(r)=ar^m+r^{m+2}\phi
\end{eqnarray}
and substitute it into (\ref{eqn:2.4}), then we get
$$\phi = \frac{1} {{r^{m + 2} }}\left\{\int_0^r {\frac{1} {s}} \int_0^s
{\big[\frac{{m^2 }} {{2t}}} \sin 2(at^m+t^{m+2}\phi) +
g(at^m+t^{m+2}\phi)t \big]dtds  - ar^m \right\}.$$
Define a map $$T:　C[0,\,\delta] \to C[0,\,\delta]$$ by
\begin{eqnarray}\label{eqn:2.6}
T(\phi)=\frac{1}{{r^{m + 2} }}\left\{\int_0^r {\frac{1}{s}} \int_0^s {\big[\frac{{m^2 }}
{{2t}}} \sin 2(at^m+t^{m+2}\phi) + g(at^m+t^{m+2}\phi)t \big]dtds  -
ar^m \right\},
\end{eqnarray} where $\delta$ would be determined later.

First, we need to verify that $T$ is well defined. Since, for any
fixed continuous function $\phi$, there holds true
\begin{eqnarray}
 &&\lefteqn{\left| {\int_0^r {\frac{1}{s}} \int_0^s {\big[\frac{{m^2 }}{{2t}}} \sin 2(at^m  + t^{m + 2} \phi ) + g(at^m  + t^{m + 2} \phi )t\big]dtds} \right| }\nn\\
 &\le& \int_0^r {\frac{1}{s}} \int_0^s {\big[\frac{{m^2 }}{{2t}}} 2(\left| a \right|t^m  + t^{m + 2} \left| \phi  \right|) + C(\left| a \right|t^m  + t^{m + 2} \left| \phi  \right|)t\big]dtds \nn\\
 &\le& C(\left| a \right| + \left\| \phi  \right\|_{C[0,\delta ]} )\int_0^r {\frac{1}{s}} \int_0^s {(t^{m - 1}  + t^{m + 3} )dtds} \nn \\
 &\le& C(\left| a \right| + \left\| \phi  \right\|_{C[0,\delta ]} )(r^m  + r^{m + 4} ) \nn\\
 &\le& C(\left| a \right| + \left\| \phi  \right\|_{C[0,\delta ]} )(\delta ^m  + \delta ^{m + 4} )<+\infty,
\end{eqnarray}
where $C$ is independent of $\phi$, therefore we know that
$T(\phi)(\,\cdot\,)$ is continuous on $(0, \delta]$. The remaining
is to verify that $T(\phi)(\,\cdot\,)$ is also continuous at $r=0$.
Indeed,
\begin{eqnarray}
\mathop {\lim }\limits_{r \to 0} T(\phi )
  &=& \mathop {\lim }\limits_{r \to 0} \frac{{\int_0^r {\frac{1}{s}} \int_0^s {[\frac{{m^2 }}{{2t}}} \sin 2(at^m  + t^{m + 2} \phi ) + g(at^m  + t^{m + 2} \phi )t]dtds - ar^m }}{{r^{m + 2} }} \nn\\
  &=& \mathop {\lim }\limits_{r \to 0} \frac{{\frac{1}{r}\int_0^r {[\frac{{m^2 }}{{2t}}} \sin 2(at^m  + t^{m + 2} \phi ) + g(at^m  + t^{m + 2} \phi )t]dt - mar^{m - 1} }}{{(m + 2)r^{m + 1} }}\nn \\
  &=& \mathop {\lim }\limits_{r \to 0} \frac{{\int_0^r {[\frac{{m^2 }}{{2t}}} \sin 2(at^m  + t^{m + 2} \phi ) + g(at^m  + t^{m + 2} \phi )t]dt - mar^m }}{{(m + 2)r^{m + 2} }} \nn\\
  &=& \mathop {\lim }\limits_{r \to 0} \frac{{\frac{{m^2 }}{{2r}}\sin 2(ar^m  + r^{m + 2} \phi ) + g(ar^m  + r^{m + 2} \phi )r - m^2 ar^{m - 1} }}{{(m + 2)^2 r^{m + 1} }} \nn\\
  &=& \mathop {\lim }\limits_{r \to 0} \frac{{\frac{{m^2 }}{2}\sin 2(ar^m  + r^{m + 2} \phi ) + g(ar^m  + r^{m + 2} \phi )r^2  - m^2 ar^m }}{{(m + 2)^2 r^{m + 2} }} \nn\\
  &=& \left\{ \begin{array}{l}
  \frac{1}{9}(\phi (0) - \frac{2}{3}a^2  + g'(0)a), \quad  \quad m = 1, \\
  \\
  \frac{1}{{(m + 2)^2 }}[m^2 \phi (0) + g'(0)a], \quad \quad m \ge 2 .
 \end{array} \right.
\end{eqnarray}

Next, we turn to discussing when $T$ is a contraction mapping. We
compute
\begin{eqnarray}\label{eqn:2.9}
  \left| {T(\phi _1 ) - T(\phi _2 )} \right|
  &\le& \frac{1}{{r^{m + 2} }}\int_0^r {\frac{1}{s}} \int_0^s {\{ \frac{{m^2 }}{{2t}}} \left| {\sin 2(at^m  + t^{m + 2} \phi _1 ) - \sin 2(at^m  + t^{m + 2} \phi _2 )} \right| \nn\\
  &&{}+ \left| {g(at^m  + t^{m + 2} \phi _1 ) - g(at^m  + t^{m + 2} \phi _2 )} \right|t\} dtds \nn\\
  &\le& \frac{1}{{r^{m + 2} }}\int_0^r {\frac{1}{s}} \int_0^s {\{ \frac{{m^2 }}{{2t}}} 2t^{m + 2} \left| {\phi _1  - \phi _2 } \right| + Ct^{m + 3} \left| {\phi _1  - \phi _2 } \right|\} dtds \nn\\
  &\le& \frac{1}{{r^{m + 2} }}\left\| {\phi _1  - \phi _2 } \right\|_{C[0,\delta ]} \int_0^r {\frac{1}{s}} \int_0^s {\{ m^2 } t^{m + 1}  + Ct^{m + 3} \} dtds \nn\\
  &\le& \frac{1}{{r^{m + 2} }}\left\| {\phi _1  - \phi _2 } \right\|_{C[0,\delta ]} \left(\frac{{m^2 }}{{(m + 2)^2 }}r^{m + 2}  + \frac{C}{{(m + 4)^2 }}r^{m + 4}\right) \nn\\
  &\le& \left(\frac{{m^2 }}{{(m + 2)^2 }} + \frac{C}{{(m + 4)^2 }}\delta ^2 \right)\left\| {\phi _1  - \phi _2 } \right\|_{C[0,\delta ]},
\end{eqnarray}
where $C$ is a positive constant independent of $\phi$. It's easy to
see that $T$ is a contraction mapping if $\delta$ is small enough
such that
$$\frac{{m^2 }}{{(m + 2)^2 }} + \frac{C}{{(m + 4)^2 }}\delta ^2
<1.$$ Thus there exists unique fix point $\phi ^* \in C[0,\delta]$
such that $T(\phi ^*)=\phi ^*$. So, $$h(r)=ar^m+r^{m+2}\phi ^*$$
satisfies (\ref{eqn:2.1}). Moreover, by a simple calculation, we can
verify that $h(r)$ also satisfies (\ref{eqn:2.2}). This means that
$h(r)$ is a local solution to (\ref{eqn:2.1})-(\ref{eqn:2.2}). Hence, by a standard argument we can extend the local solution to a global solution. Thus, we have shown that the following theorem holds true.

\begin{thm}\label{thm:2.1}
Assume that $g(x) \in C^\infty(\mathbb{R})$ satisfies $ \|g\|_{C^1(\mathbb{R})}<\infty$. Then, (\ref{eqn:2.1})-(\ref{eqn:2.2})
always admits a unique global solution.
\end{thm}

For simplicity, we denote the solution of
(\ref{eqn:2.1})-(\ref{eqn:2.2}) by $h_a(r)$ to emphasize the
dependence on the initial value $a$. Now, we discuss the continuous
dependence on the initial data of these obtained solutions. We need to
establish the following theorem.

\begin{thm}\label{thm:2.2}
Assume that $g(x) \in C^\infty(\mathbb{R})$ satisfies $ \|g\|_{C^1(\mathbb{R})}<\infty$.
If $h_{a_0}(r)$ is the solution of
(\ref{eqn:2.1})-(\ref{eqn:2.2}), then, $\forall R>0$,
$\forall \varepsilon>0$, there exists
$\eta=\eta(a_0,\varepsilon,R)>0$ such that, if $\left|a-a_0
\right|<\eta$, there holds
\begin{eqnarray}\label{eqn:2.10}
\left\| h_a(r)-h_{a_0}(r) \right\|_{C^1[0,R]} \le \varepsilon.
\end{eqnarray}
\end{thm}

\begin{proof}
By the standard O.D.E. theory we know that the solutions to
(\ref{eqn:2.1}) depend continuously on the initial data. We need
only to prove the theorem in the case $R=\delta$, which has been
determined in (\ref{eqn:2.9}).

From the proof of Theorem \ref{thm:2.1}, we know that $h_a$ is
of the following form:
\begin{eqnarray}\label{Ex}
h_a(r)=ar^m+r^{m+2}\phi_a.
\end{eqnarray}
In order to get (\ref{eqn:2.10}), we need to estimate size of
$\left\| \phi_a(r)-\phi_{a_0}(r) \right\|_{C[0,R]}.$ Indeed,
\begin{eqnarray}
\left| {\phi _a  - \phi _{a_0 } } \right| &=& \left| {T(\phi _a ) - T(\phi _{a_0 } )} \right| \nn\\
  &\le& \frac{1}{{r^{m + 2} }}\int_0^r {\frac{1}{s}} \int_0^s {\{ \frac{{m^2 }}{{2t}}} \left| {(\sin 2(at^m  + t^{m + 2} \phi _a ) - \sin 2(a_0 t^m  + t^{m + 2} \phi _{a_0 }
  )- 2(a - a_0 )t^m } \right| \nn\\
  &&+ \left| {g(at^m  + t^{m + 2} \phi _a ) - g(a_0 t^m  + t^{m + 2} \phi _{a_0 } )} \right|t\} dtds \nn\\
  &\le& \frac{1}{{r^{m + 2} }}\int_0^r {\frac{1}{s}} \int_0^s {\{ \frac{{m^2 }}{{2t}}} [2t^{m + 2} \left| {\phi _a  - \phi _{a_0 } } \right| + (C + (a - a_0 )^2
  )t^{2m}\nn\\
  &&{}+(C + \left| {\phi _a  - \phi _{a_0 } } \right|^2 )t^{2m + 4} ]
   + C(\left| {a - a_0 } \right|t^{m + 1}  + t^{m + 3} \left| {\phi _a  - \phi _{a_0 } } \right|)\}
   dtds\nn\\
  &\le& \frac{1}{{r^{m + 2}}}\left\{ \left\| {\phi _a  - \phi _{a_0 } } \right\|_{C[0,\delta ]}
  \left( {\frac{{m^2 }}{{(m + 2)^2 }} + C\delta ^2 } \right)r^{m +
  2}\right.\nn\\
  && \left. + C\left| {a - a_0 } \right|\frac{1}{{(m + 2)^2 }}r^{m + 2} \right\},
\end{eqnarray}
where $C$ is a positive constant depending only on $m$ and $a_0$. By
choosing $\delta$ small enough such that
 $$\frac{{m^2 }}{{(m + 2)^2}} + C\delta ^2 <1,$$ we can derive
\begin{eqnarray}\label{eqn:2.13}
\left\| {\phi _a  - \phi _{a_0 } } \right\|_{C[0,\delta]} \le C\left| {a - a_0 } \right|,
\end{eqnarray}
where $C$ is a positive constant depending on $m$ and $a_0$. It follows that
\begin{eqnarray}\label{eqn:2.14}
\left| h_a(r)-h_{a_0}(r) \right| &\le& \left| a-a_0 \right| r^m +r^{m+2} \left| \phi_a -\phi_{a_0} \right|\nn\\
&\le& \left| a-a_0 \right| \delta^m +C\delta^{m+2} \left| \phi_a -\phi_{a_0} \right|\nn\\
&\le& C\left| a-a_0 \right|,
\end{eqnarray}
and
\begin{eqnarray}\label{eqn:2.15}
 {}\left| {h_a '(r) - h_{a_0 } '(r)} \right|
  &\le& \frac{1}{r}\int_0^r {\{ \frac{{m^2 }}{{2t}}} \left| {(\sin 2(at^m  + t^{m + 2} \phi _a ) - \sin 2(a_0 t^m  + t^{m + 2} \phi _{a_0 } )} \right| \nn\\
  &&{}+ \left| {g(at^m  + t^{m + 2} \phi _a ) - g(a_0 t^m  + t^{m + 2} \phi _{a_0 } )} \right|t\} dt \nn\\
  &\le& \frac{1}{r}\int_0^r {\{ \frac{{m^2 }}{{2t}}} [2\left| {a - a_0 } \right|t^m  + 2t^{m + 2} \left| {\phi _a  - \phi _{a_0 } } \right|] \nn\\
  &&+ C(\left| {a - a_0 } \right|t^{m + 1}  + t^{m + 3} \left| {\phi _a  - \phi _{a_0 } } \right|)\} dt \nn\\
  &\le& \frac{1}{r}\left\{ \left| {a - a_0 } \right|\left( mr^m  + C\frac{1}{{(m + 2)^2 }}r^{m + 2}\right)\right. \nn\\
  & & \left. + \left\| {\phi _a  - \phi _{a_0 } } \right\|_{C[0,\delta ]} (\frac{{m^2 }}{{m + 2}}r^{m + 2}  + \frac{C}{{m + 4}}r^{m + 4} )\right\}  \nn\\
  &\le& C\left| {a - a_0 } \right|,
\end{eqnarray}
where $C$ is a positive constant depending only on $m$ and $a_0$.
Thus, we obtain the desired estimate:
$$\left\| h_a(r)-h_{a_0}(r) \right\|_{C^1[0,\delta]} \le C\left| a-a_0 \right| .$$
\end{proof}

\begin{rem}
By the standard elliptic regularity theory, it is not difficult for us to see that, the conclusions in Theorem \ref{thm:2.1} and Theorem \ref{thm:2.2} also hold true, if $g \in C^{\alpha(m)}(\mathbb{R})$, where $\alpha(m)=max\{1,|m|-2\}$.

\end{rem}

\section{Qualitative Analysis of the O.D.E}
In this section, we will establish a series of lemmas to
characterize the behavior of solutions to (\ref{eqn:2.1}) under some
suitable assumptions on the function $g$.

First, let's recall the Pohozaev identity of (\ref{eqn:2.1}).
By multiplying the both sides of equation (\ref{eqn:2.3}) by $rh'(r)$
and integrating from $s$ to $r$, we obtain the Pohozaev identity
\begin{eqnarray}\label{eqn:3.1}
\left( {rh'(r)} \right)^2  - \left( {sh'(s)} \right)^2  = m^2 [\sin
^2 h(r) - \sin ^2 h(s)] + 2\int_s^r {g(h(t))h'(t)t^2 dt},
\end{eqnarray}
or
\begin{eqnarray}\label{eqn:3.2}
\left( {rh'(r)} \right)^2  - \left( {sh'(s)} \right)^2
&=& m^2 [\sin ^2 h(r) - \sin ^2 h(s)] + 2[G(h(r))r^2  - G(h(s))s^2 ] \nn\\
& & -4\int_s^r {G(h(t))tdt},
\end{eqnarray}
where $$G(x)=-\int_x^\pi g(t) dt.$$

\begin{lem}
Assume that $g(x)\in C^\infty([0,\, \pi])$ satisfies $(i)-(iii)$ and $h(r)$ satisfies (\ref{eqn:2.1}). If there exists
$r_0 \in (0, +\infty)$ such that $$0 \le h(r_0) \le \min
\{\pi-\xi,\,\xi\}\quad\text{and} \quad h'(r_0)>0,$$ then there exists $r_1>r_0$
such that $h'(r)>0$ for any $r \in [r_0,\,r_1]$ and $ h(r_1)=\xi$.
\end{lem}

\begin{proof}
By the Pohozaev identity, we have
\begin{eqnarray}\label{eqn:3.3}
\left( {rh'(r)} \right)^2 = \left( {r_0h'(r_0)} \right)^2 + m^2 [\sin ^2 h(r) - \sin ^2 h(r_0)] + 2\int_{r_0}^r {g(h(t))h'(t)t^2 dt}.
\end{eqnarray}
Let $$r^*=\sup\{s \in [r_0,+\infty) \ |\  h'(r)>0,\ r \in [r_0,s)\}.$$
It is easy to see that $r_0<r^* \le +\infty$, since $h'(r_0)>0$.

We claim that
\begin{eqnarray}\label{eqn:3.4}
h(r^*)>\xi.
\end{eqnarray}
If this was false, then there existed some $r \in [r_0, r^*)$ such that $h(r) \le \xi.$
Hence, from the assumptions on $g(x)$ we have, for
$r \in [r_0,\, r^*)$, there holds
\begin{eqnarray}\label{eqn:3.5}
\sin^2 h(r) \ge \sin^2 h(r_0)\quad\text{and}\quad g(h(r)) \ge 0.
\end{eqnarray}
Combining (\ref{eqn:3.3}) and (\ref{eqn:3.5}) we obtain
\begin{eqnarray}\label{eqn:3.6}
rh'(r) \ge r_0h'(r_0)>0, \quad \forall r \in [r_0,\, r^*).
\end{eqnarray}
It follows that $$h(r)=\int_{r_0}^r {h'(t)dt}+h(r_0) \ge r_0h'(r_0)\int_{r_0}^r {\frac{1}{t} dt}+h(r_0).$$
This implies that $r^*<+\infty$. Otherwise, we would deduce that $h(r)$ is unbounded in the interval $[r_0,+\infty)$,
a contradiction.

By the definition of $r^*$, we get $h'(r^*)=0$ which contradicts (\ref{eqn:3.6}). Thus, we show the assertion.
Therefore, we can choose $r_1 \in (r_0,r^*)$ such that $h(r_1)=\xi$. By the definition of $r^*$, we get
$$h'(r)>0,\quad\forall r \in [r_0,r_1].$$
\end{proof}

\begin{coro}\label{coro:3.2}
Suppose that $g(x)$ satisfies $(i)-(iii)$. If $h_a(r)$ is the solution of (\ref{eqn:2.1})-(\ref{eqn:2.2}) with $a>0$, then there exists
$s_a \in (0,+\infty)$ such that $h(r)$ increases monotonically from $0$ to $\xi$ on the interval $[0,s_a]$.
\end{coro}

\begin{lem}\label{lem:3.3}
Suppose that $g(x)$ satisfies $(i)-(iii)$. If $h(r)$, which is not a constant, satisfies (\ref{eqn:2.1})
on the interval $(r_0,+\infty),$  $ \xi \le h(r) \le \pi$ for any $r\in
[r_0,+\infty)$ and $\mathop {\lim }\limits_{r \to \infty
}h(r)=l>\xi$, then, there hold that for any $r\in [r_0,+\infty)$
$$h'(r)>0\quad \quad \text{and} \quad\quad l=\pi.$$
Moreover, $h(r)$ converges to $\pi$ exponentially as $r \to +\infty$.
\end{lem}

\begin{proof}
As $\pi$ is a trivial solution of equation (\ref{eqn:2.1}) and $h(r)$ is not a constant function, then
we have
\begin{eqnarray}\label{eqn:3.7}
h(r)<\pi,\quad r \in (r_0,+\infty).
\end{eqnarray}
In fact, if there exists a $r_1 \in (r_0,+\infty)$ such that $h(r_1)=\pi$, by the condition of
$ \xi \le h(r) \le \pi$ for any $r \in [r_0,+\infty)$, we get $h'(r_1)=0$. The uniqueness of
solutions to initial value problem tells that $h(r) \equiv \pi$, which contradicts the fact
that $h(r)$ is not a constant function.

Since $\mathop {\lim }\limits_{r \to \infty }h(r)=l>\xi$, then, there exists $r_1 \in [r_0,+\infty)$
such that, for any $r \in [r_1,+\infty)$, we have
 $$\frac{m^2 }{r^2} \sin h(r) \cos h(r)  + g(h(r))<0.$$
From (\ref{eqn:2.3}) we have that, for $r_1 \le s \le r < +\infty$,
\begin{eqnarray}\label{eqn:3.8}
rh'(r)=sh'(s) + \int_s^r {\left[\frac{m^2 }{t^2} \sin h \cos h  + g(h)\right]tdt}
\end{eqnarray}
and
\begin{eqnarray}\label{eqn:3.9}
h(r)=h(s)+\int_s^r{h'(t)dt}.
\end{eqnarray}
Since $ \xi \le h(r) \le \pi$ for any $r \in [r_0,+\infty)$, (\ref{eqn:3.9})
implies that there exists $C>0$ and $r_k \to +\infty$ such that
 $$-h'(r_k) \le \frac{C}{r_k}.$$
Set
$$A(r)=\int_s^r {\left[\frac{m^2 }{t^2} \sin h \cos h  + g(h)\right]tdt},$$
it follows from (\ref{eqn:3.8}) that $A(r_k) \ge -C$.
On the other hand, $A(r)$ is decreasing and negative since the integrand
is negative on the interval $[r_1,+\infty)$, we see that
$\mathop {\lim }\limits_{r \to \infty } A(r)$ exists. It implies that
$$\mathop {\lim }\limits_{r \to \infty } \frac{m^2 }{r^2} \sin h \cos h  + g(h)=g(l)=0.$$
Then we get $l=\pi.$

By (\ref{eqn:3.8}) we infer that $\mathop {\lim }\limits_{r \to \infty } rh'(r)=l_0$ exists.
Then, we can easily see that $l_0$ is zero. Otherwise, $h(r)$ would be unbounded by (\ref{eqn:3.9}).
Let $r \to +\infty$ in (\ref{eqn:3.8}) and replace $s$ by $r$, we get
\begin{eqnarray}\label{eqn:3.10}
rh'(r)=- \int_r^{+\infty} {\left[\frac{m^2 }{t} \sin h \cos h  + g(h)t\right]dt}.
\end{eqnarray}
From (\ref{eqn:3.10}) and (\ref{eqn:3.7}) we can deduce that for any $r \in [r_1,+\infty)$
$$h'(r)>0.$$
By the Pohozaev identity (\ref{eqn:3.2}), we have that, for $r_0 \le s \le r < +\infty$,
\begin{eqnarray}
{}\left( {rh'(r)} \right)^2  - \left( {sh'(s)} \right)^2
&=& m^2 [\sin ^2 h(r) - \sin ^2 h(s)] + 2[G(h(r))r^2  - G(h(s))s^2 ]  \nn\\
& &- 4\int_s^r {G(h(t))tdt}.
\end{eqnarray}
Since $\mathop {\lim }\limits_{r \to \infty }rh'(r)=0$ and $\mathop {\lim }\limits_{r \to \infty }h(r)=\pi$, we get
$$0<\int_s^{+\infty} {G(h(t))tdt}<+\infty.$$
Replacing $s$ by $r$ and then letting $r \to +\infty$ in the above identity, we obtain
\begin{eqnarray}
\left( {rh'(r)} \right)^2 = m^2 \sin ^2 h(r) + 2G(h(r))r^2 + 4\int_r^{+\infty} {G(h(t))tdt}.
\end{eqnarray}
Since $\xi \le h(r) <\pi$ for $r \in [r_0,+\infty)$, from the above identity,
we get $h'(r) \ne 0$ for any $r \in [r_0,+\infty) $. As $h'(r)$ is
continuous on the interval $[r_0,+\infty)$ and $h'(r)>0$ for any $r \in [r_1,+\infty)$,
we obtain $h'(r)>0$ for any $r \in [r_0,+\infty)$.

Now, we are in the position to prove that $h(r)$ converges to $\pi$ exponentially as $r \to +\infty$.
Let $$\widetilde h(r)=\pi-h(r).$$
Then, $\widetilde h(r)>0$ on $(0,+\infty)$ satisfies the following equation
\begin{eqnarray}\label{eqn:w1}
\widetilde h''+{1 \over r}\widetilde h'=\frac{m^2\sin 2\widetilde h}{2r^2}-g(\pi-\widetilde h).
\end{eqnarray}
Let $f(r)=be^{-\epsilon r}$. Then, it is easy to verify that $f(r)$ satisfies the following equation
\begin{eqnarray}\label{eqn:w2}
f''+{1 \over r}f'=(\epsilon^2-{\epsilon \over r})f.
\end{eqnarray}

Denote $$\beta(r)=f(r)-\widetilde h(r).$$
Then, it follows from (\ref{eqn:w1}) and (\ref{eqn:w2})
$$\beta''+{1 \over r}\beta'=(\epsilon^2-{\epsilon \over r})f-\frac{m^2\sin 2\widetilde h}{2r^2}+g(\pi-\widetilde h).$$

Since $\lim\limits_{r \to +\infty}\widetilde h(r)=0$,
we choose $R_0>0$ such that $$g(\pi-\widetilde h(r))<-{1 \over 2}g'(\pi)\widetilde h(r),\quad r \ge R_0.$$
By choosing $\epsilon=\sqrt{{1 \over 2}g'(\pi)}$ and $b=b_0$ such that $$\beta(R_0)=b_0 e^{-\epsilon R_0}-\widetilde h(R_0)>0,$$
we have that, for $r\ge R_0$,
$$\beta''+{1 \over r}\beta'=(\epsilon^2-{\epsilon \over r})f-\frac{m^2\sin 2\widetilde h}{2r^2}+g(\pi-\widetilde h)<{1 \over 2}g'(\pi)\beta$$
i.e. $$\beta''+{1 \over r}\beta'-{1 \over 2}g'(\pi)\beta<0, \quad r \in [R_0,+\infty).$$
Since $\beta(R_0)>0$ and $\lim\limits_{r \to +\infty}\beta(r)=0$, by the maximum principle
we have $\beta(r)>0$ as $r\ge R_0$, i.e. $$0<\widetilde h(r)=\pi-h(r)<b_0 e^{-\sqrt{{1 \over 2}g'(\pi)}r}, \quad r \in [R_0,+\infty).$$
Thus, we complete the proof.
\end{proof}

\begin{lem}\label{lem:3.3_1}
Suppose that $g(x)\in C^\infty([0,\,\pi])$ satisfies $(i)-(iii)$. If $h(r)$ satisfies (\ref{eqn:2.1})
on the interval $(s,\,\mu)$ with $ \xi=h(s) \le h(r) \le \pi=h(\mu)$, $r \in (s,\,\mu)$,
then there holds true $$ h'(r)>0,\quad r \in [s,\,\mu].$$
\end{lem}

\begin{proof}
Since $\pi$ is a trivial solution to (\ref{eqn:2.1}), we get $$ h(r)<\pi,\  r \in (s,\mu)\quad \text{and}\quad h'(\mu)>0.$$
By the Pohozaev identity, we have that, for $r \in [s,\mu]$,
\begin{eqnarray}
\left( {rh'(r)} \right)^2 =  \left( {\mu h'(\mu)} \right)^2 + m^2 \sin ^2 h(r) + 2G(h(r))r^2 + 4\int_r^{\mu} {G(h(t))tdt}.
\end{eqnarray}
Since $G(x) \ge 0$ for any $x \in [\xi,\pi]$, we have $h'(r) \neq 0$ for any  $r\in [s, \mu]$.
As $h'(r)$ is continuous on the interval $[s,\mu]$ and $h'(\mu)>0$, we obtain $h'(r) > 0$ for any $r\in [s,\mu].$

\end{proof}

\begin{lem}\label{lem:3.4}
Suppose that $g(x)$ satisfies $(i)-(iii)$. Let $h_i(r)$, $i=1,2$, be two increasing functions satisfying (\ref{eqn:2.1}).
If they intersect with each other at two different points in the domain $(0,+\infty) \times [\xi,\pi]$,
then $h_1 \equiv h_2$ on the interval $(0,+\infty)$.
\end{lem}

\begin{proof}
If this lemma is false, without loss of generality, we assume there exist $0< r_1 <r_2<+\infty$ such that
$h_1(r_i)=h_2(r_i)$, where $i=1, 2$, and
\begin{eqnarray}\label{eqn:3.12}
 h_1(r)>h_2(r)\quad\text{for any}\quad r \in (r_1,r_2).
\end{eqnarray}
Then, we get
\begin{eqnarray}\label{eqn:3.13}
h_1'(r_1) > h_2'(r_1)\ge 0,\quad 0\le  h_1'(r_2) < h_2'(r_2).
\end{eqnarray}

By the Pohozaev identity (\ref{eqn:3.2}),  for $i=1,2$, we have
\begin{eqnarray}
{}\left( {r_2h_i'(r_2)} \right)^2  - \left( {r_1h_i'(r_1)} \right)^2
&=& m^2 [\sin ^2 h_i(r_2) - \sin ^2 h_i(r_1)] + 2[G(h_i(r_2))r_2^2\nn\\
& &- G(h_i(r_1)r_1^2 ] - 4\int_{r_1}^{r_2} {G(h_i(t))tdt}.
\end{eqnarray}
Substituting $h_1(r_i)=h_2(r_i),i=1,2$, into the above identities, we can obtain
\begin{eqnarray}\label{eqn:3.15}
r_2^2(h_1'^2(r_2)-h_2'^2(r_2))+r_1^2(h_2'^2(r_1)-h_1'^2(r_1))=4\int_{r_1}^{r_2} {[G(h_2(t))-G(h_1(t))]tdt}.
\end{eqnarray}
From (\ref{eqn:3.13}) we can see that the left hand side of (\ref{eqn:3.15}) is negative.
However, by (\ref{eqn:3.12}) and the fact $G(x)$ is decreasing on the interval $[\xi,\,\pi]$,
we infer that the right hand side of (\ref{eqn:3.15}) is positive. There exists a contradiction.
The lemma is proved.

\end{proof}

\section{The Solvability of the boundary value problem of O.D.E.}

In order to prove Theorem \ref{thm:1}, we need to study the solvability of the boundary
value problem of (\ref{eqn:2.1}) to look for some comparing functions in this section.
We will employ the variational method to approach the existence of such a boundary value problem.
The following argument is the key ingredient of the proof of Theorem \ref{thm:1}.

For $0<s<\mu<+\infty$, we consider the following two problems:
\begin{eqnarray}\label{P_s} (P_s)
\left\{
\begin{aligned}
&(rh')' = \frac{{m^2 }}{r}\sin h \cos h  + g(h)r, \quad \xi<h(r)<\pi,\quad r \in (s,+\infty),\\
&h(s)=\xi, \quad \mathop {\lim }\limits_{r \to +\infty }h(r)=\pi,\nn\\
\end{aligned}\right.
\end{eqnarray}
and
\begin{equation}\label{Q_s} Q_{(s,\mu)}
\left\{
\begin{aligned}
&(rh')' = \frac{{m^2 }}{r}\sin h \cos h  + g(h)r, \quad \xi<h(r)<\pi,\quad r \in (s,\mu),\\
&h(s)=\xi, \quad h(\mu)=\pi.\nn\\
\end{aligned}\right.
\end{equation}

\begin{thm}\label{thm:4.1}
Suppose that $g(x)\in C^\infty([0,\,\pi])$ satisfies $(i)-(iii)$. Then,there exists a unique solution $\widetilde h(r)$ to $(P_s)$.
Moreover, $\widetilde h'(r)>0$, $\forall r \in [s,+\infty)$.
\end{thm}

\begin{proof}
\emph{Step 1}. First, we consider the functional $J_s$ given by
$$J_s(h)=\frac{1}{2}\int_s^{+\infty}{\left[(h'(r))^2+\frac{{m^2}}{{r^2}}\sin^2 h(r)\right]rdr} + \int_s^{+\infty}{G(h(r))rdr},$$
which is defined on the space $X_s$ given by
$$X_s=\{h(r)\ |\  \pi-h(r) \in H^1([s,+\infty),rdr),\  \xi \le h(r) \le \pi,\  h(s)=\xi\}.$$
By the definition of $G(x)$, it is easy to check that there exists a
constant $C>0$ such that
$$C^{-1}G(x) \le (\pi-x)^2 \le CG(x),\  x \in [\xi,\pi].$$
Hence, we know $J_s$ is well defined in the space $X_s$.

It is easy to see that, for any $h(r) \in X_s$, $h(r)$ is continuous on the interval
$[s,+\infty)$ and $\mathop {\lim }\limits_{r \to \infty }h(r)=\pi$.
Moreover, the fact $G(h(r))\ge 0$ for any $r \in [s,+\infty)$
implies that $J_s(h(r)) \ge0$.

Choose a minimizing sequence $\{h_k\} \subseteq X_s$ such that $$\mathop {\inf }\limits_{X_s } J_s
=\mathop {\lim }\limits_{k \to \infty }J_s(h_k).$$
Since
\begin{eqnarray}
\int_s^{+\infty}{[((\pi-h_k)')^2+(\pi-h_k)^2]rdr}&=&\int_s^{+\infty}{[(h_k')^2+(\pi-h_k)^2]rdr}\nn\\
&\le& \int_s^{+\infty}{(h_k')^2rdr}+C\int_s^{+\infty}{G(h_k)rdr}\nn\\
&\le& CJ_s(h_k) \le C,
\end{eqnarray}
where $C$ is independent of $k$, up to a subsequence, there exists $\pi -\widetilde h \in H^1([s,+\infty),rdr)$ such that
\begin{eqnarray}\label{eqn:4.4}
\pi -h_k \to \pi -\widetilde h \quad \text{weakly in}\quad
H^1([s,+\infty),rdr),
\end{eqnarray}
and
\begin{eqnarray}\label{eqn:4.5}
\forall R \in (s,+\infty),\ \pi -h_k \to \pi -\widetilde h \quad
\text{in} \quad C[s,R].
\end{eqnarray}
By (\ref{eqn:4.5}), we get $\widetilde h(s)=\xi$ and $\mathop {\lim
}\limits_{k \to \infty }h_k(r)=\widetilde h(r)$ for any $r\in
[s,+\infty)$. So $\widetilde h \in X_s$. By (\ref{eqn:4.4}), we
obtain
$$\int_s^{+\infty}{({\widetilde h}')^2rdr} \le \mathop {\varliminf }\limits_{k \to \infty } \int_s^{+\infty}{(h_k')^2rdr}.$$
By Fatou lemma and  $\mathop {\lim }\limits_{k \to \infty
}h_k(r)=\widetilde h(r)$ for any $r \in [s,+\infty)$, we have
$$0\le \int_s^{+\infty}{\frac{{m^2}}{{r^2}}\sin^2 \widetilde h(r)rdr} \le \mathop {\varliminf }\limits_{k \to \infty } \int_s^{+\infty}{\frac{{m^2}}{{r^2}}\sin^2 h_k(r)rdr},$$
$$0 \le \int_s^{+\infty}{G(\widetilde h)rdr} \le \mathop {\varliminf }\limits_{k \to \infty } \int_s^{+\infty}{G(h_k)rdr}.$$
Immediately, from the above three inequalities we obtain
$$J_s(\widetilde h) \le \mathop {\varliminf }\limits_{k \to \infty
}J_s(h_k)=\mathop {\inf }\limits_{X_s } J_s.$$
Since $\widetilde h \in X_s$, we know that $\widetilde h$ is a
minimal point of $J_s$ in $X_s$, i.e.
$$J_s(\widetilde h)=\mathop {\inf }\limits_{X_s } J_s.$$

\emph{Step 2.} Next, we need to verify that $\widetilde h$ is just the
solution to the boundary value problem ($P_s$). Obviously,
$\widetilde h$ satisfies the boundary conditions of $(P_s)$, we need
only to prove $\widetilde h$ satisfies (\ref{eqn:2.1}) on the
interval $(s,+\infty)$.

We say $\varphi \in C_0^{\infty}(s,+\infty)$ is a admissible
variational function for $\widetilde h$,
if there exists $\varepsilon >0$ such that $\widetilde h+t\varphi
\in X_s,\ t \in [0,\varepsilon)$.

Since $\widetilde h$ is the minimal point of
$J_s(\,\cdot\,)$, then, for any admissible variational function $\varphi$, we have
\begin{eqnarray}\label{eqn:4.6}
\left. {\frac{d}{{dt}}J_s (\widetilde h  + t\varphi )} \right|_{t =
0}  \ge 0.
\end{eqnarray}
More precisely,
\begin{eqnarray}\label{eqn:4.7}
\int_s^{+\infty}{({\widetilde h}' \varphi'+\frac{{m^2}}{{r^2}}\sin \widetilde h \cos \widetilde h \varphi)rdr} + \int_s^{+\infty}{g(\widetilde h)\varphi rdr} \ge 0.
\end{eqnarray}

(i). If $\xi<\widetilde h<\pi,\ r \in (r_1,r_2)$, where $s \le
r_1<r_2 \le +\infty$, by (\ref{eqn:4.7}), it's easy to check
$\widetilde h$ satisfies (\ref{eqn:2.1}) on the interval
$(r_1,r_2)$.\\

(ii). We claim that $\widetilde h(r)< \pi$ for any $r \in (s,+\infty)$.\\
If the assertion was false, we would obtain a contradiction. We
define $$r^*=\inf\{r \in (s,+\infty)| \,\, \widetilde h(r)=\pi\}.$$
From the definition of $X_s$ we can easily see that $s<r^*<+\infty$.
We still need to define the following continuous function
$$\widehat h(r) = \left\{ \begin{array}{l}
 \widetilde h(r),\quad r<r^* ,\\
 \pi ,\quad r \ge r^*.
 \end{array} \right.$$
Obviously, $\widehat h \in X_s$. It is easy to see that, if
$\widehat h \neq \widetilde h$, from the definition of $J_s$ and
$\widehat h$ we infer
$$J_s(\widehat h) < J_s(\widetilde h).$$ This contradicts the fact $J_s(\widetilde
h)=\mathop {\inf }\limits_{X_s } J_s$. This means that $\widehat h
\equiv \widetilde h$.

By the definition of $r^*$, we choose small $\delta>0$ such that
$$\widetilde h(r)<\pi, \quad\forall r \in [r^*-\delta,\,r^*).$$
So, by the conclusion of (i), we obtain that $\widetilde h(r)$ satisfies
(\ref{eqn:2.1}) in the interval $(r^*-\delta,\, r^*)$
and $$ \widetilde h'(r_-^*)=\mathop {\lim }\limits_{r \to r^*_- }\widetilde h'(r) >0.$$

We choose $\varphi \in C_0^{\infty}[r^*-\delta,\, r^*+\delta]$ such
that $\varphi \le 0,\ \varphi(r^*)=-1$. Then $\varphi$ is an admissible variational function
for $\widetilde h$. By (\ref{eqn:4.7}), we have
\begin{eqnarray}
0&\le& \int_s^{+\infty}{({\widetilde h}' \varphi'+\frac{{m^2}}{{r^2}}\sin \widetilde h \cos \widetilde h \varphi)rdr} + \int_s^{+\infty}{g(\widetilde h)\varphi rdr}\nn\\
&=&\int_{r^*-\delta}^{r^*}{({\widetilde h}' \varphi'+\frac{{m^2}}{{r^2}}\sin \widetilde h \cos \widetilde h \varphi)rdr} + \int_{r^*-\delta}^{r^*}{g(\widetilde h)\varphi rdr}\nn\\
&=&\left. \widetilde h'(r) \varphi(r) r \right|_{r^*-\delta}^{r^*}+\int_{r^*-\delta}^{r^*}{[-(r{\widetilde h}')' +\frac{{m^2}}{{r}}\sin \widetilde h \cos \widetilde h +g(\widetilde h)r]\varphi dr}\nn\\
&=&-\widetilde h'(r_-^*)r^*<0,
\end{eqnarray}
there exists a contradiction. So, $\widetilde h(r)< \pi$ for any $r
\in (s,+\infty).$\\

(iii). We claim that $\widetilde h(r)> \xi,\ r \in (s,+\infty)$.\\
If there exist $r \in (s,+\infty)$ such that $\widetilde h(r)=\xi$,
then we define $\widehat r=\sup\{r \in (s,+\infty)\ |\ \widetilde
h(r)=\xi\}$, and obtain $s<\widehat r<+\infty$.

By the definition of $\widehat r$ and the conclusion of (ii), we get
$$\xi<\widetilde h(r)<\pi,\ r \in (\widehat r,+\infty).$$ So, by the
conclusion of (i), we obtain that $\widetilde h(r)$ satisfies
(\ref{eqn:2.1}) in the interval $(\widehat r,+\infty)$. Since
$\mathop {\lim }\limits_{r \to \infty }\widetilde h(r)=\pi$, by
Lemma \ref{lem:3.3}, we get $$ \widetilde h'(\widehat r_+)=\mathop
{\lim }\limits_{r \to \widehat r_+ }\widetilde h'(r) >0.$$

\noindent We need to consider the following two cases:

Case $\rm I$: There exists small $\delta >0$ such that $$ \widetilde
h(r) \equiv \xi,\quad\forall r \in [\widehat r-\delta,\, r^*).$$ For this case, we
choose $\varphi \in C_0^{\infty}[\widehat r-\delta, \,\widehat
r+\delta]$ such that $0 \le \varphi \le 1,\ \varphi(\widehat r)=1$.
Then $\varphi$ is a admissible variational function for $\widetilde h$. By (\ref{eqn:4.7}),
\begin{eqnarray}
0&\le& \int_s^{+\infty}{({\widetilde h}' \varphi'+\frac{{m^2}}{{r^2}}\sin \widetilde h \cos \widetilde h \varphi)rdr} + \int_s^{+\infty}{g(\widetilde h)\varphi rdr}\nn\\
&=&\int_{\widehat r-\delta}^{\widehat r}{\frac{{m^2}}{{r^2}}\sin \widetilde h \cos \widetilde h \varphi rdr} +\left. \widetilde h'(r) \varphi(r) r \right|_{\widehat r}^{\widehat r+\delta}\nn\\
&=& m^2 \sin \xi \cos \xi \int_{\widehat r-\delta}^{\widehat r}{\frac{\varphi}{r}dr}-\widetilde h'(\widehat r_+)\widehat r\nn\\
&\le& m^2 \log {\frac {\widehat r}{\widehat r-\delta}} -\widetilde h'(\widehat r_+)\widehat r.
\end{eqnarray}
we can choose $\delta$ small enough such that $$m^2 \log {\frac
{\widehat r}{\widehat r-\delta}} -\widetilde h'(\widehat
r_+)\widehat r<0.$$  Obviously, this is a contradiction.

Case $\rm II$: There exists small $\delta >0$ such that $$\xi
<\widetilde h(r)<\pi,\quad\forall r \in [\widehat r-\delta,\,\widehat r).$$ For
this case, by the conclusion of (i), we obtain that $\widetilde
h(r)$ satisfies (\ref{eqn:2.1}) in the interval $(\widehat
r-\delta,\,\widehat r)$. Moreover, $ \widetilde h'(\widehat
r_-)=\mathop {\lim }\limits_{r \to \widehat r_- }\widetilde h'(r)
\le 0.$

We can choose $\varphi \in C_0^{\infty}[\widehat r-\delta,\, \widehat r+\delta]$
such that $0 \le \varphi \le 1,\ \varphi(\widehat r)=1$. Then $\varphi$ is a
admissible variational function for $\widetilde h$.
By (\ref{eqn:4.7}),
\begin{eqnarray}
0&\le& \int_s^{+\infty}{({\widetilde h}' \varphi'+\frac{{m^2}}{{r^2}}\sin \widetilde h \cos \widetilde h \varphi)rdr} + \int_s^{+\infty}{g(\widetilde h)\varphi rdr}\nn\\
&=&\left. \widetilde h'(r) \varphi(r) r \right|_{\widehat r-\delta}^{\widehat r} +\left. \widetilde h'(r) \varphi(r) r \right|_{\widehat r}^{\widehat r+\delta}\nn\\
&=& (\widetilde h'(\widehat r_-)-\widetilde h'(\widehat r_+))\widehat r<0,
\end{eqnarray}
there exists a contradiction.

From the above arguments on two cases, we know that there always
holds true $\widetilde h(r)> \xi$ for any $ r \in (s,+\infty)$.

Combining (i), (ii) and (iii), we get $\xi <\widetilde h(r)<\pi$
satisfies (\ref{eqn:2.1}) in the interval $(s,+\infty)$.\\

\emph{Step 3.} By Lemma \ref{lem:3.3}, we immediately know that there holds true
$$\widetilde
h'(r)>0$$ for $r \in [s,+\infty)$.

The remaining is to prove the uniqueness of solution to $(P_s)$.
Assume $\widetilde h_1(r)$ and
$\widetilde h_2(r)$ are two different solutions of $(P_s)$. By the
uniqueness of initial value problem, we get $\widetilde h_1'(s)\ne
\widetilde h_2'(s)$. Without loss of generality, we assume
$\widetilde h_1'(s)> \widetilde h_2'(s)$.

By Lemma \ref{lem:3.4}, we get
\begin{eqnarray}\label{eqn:4.11}
\widetilde h_1(r)> \widetilde h_2(r),\quad r \in (s,+\infty).
\end{eqnarray}
By the Pohozaev identity (\ref{eqn:3.2}), for $i=1,2$, we have
\begin{eqnarray}\label{eqn:4.12}
&&{}\left( {s\widetilde h_i'(s)} \right)^2 = m^2 \sin ^2 \xi  +
2G(\xi)r^2 + 4\int_s^{+\infty} {G(\widetilde h_i(t))tdt}.
\end{eqnarray}
Hence, from (\ref{eqn:4.11}) and (\ref{eqn:4.12}) we deduce that
\begin{eqnarray}\label{eqn:4.13}
0<\left( {s\widetilde h_1'(s)} \right)^2-\left( {s\widetilde h_2'(s)} \right)^2=4\int_s^{+\infty} {[G(\widetilde h_1)-G(\widetilde h_2)]tdt}.
\end{eqnarray}
However, since $G(x)$ is decreasing on the interval $[\xi,\,\pi]$, by
(\ref{eqn:4.11}) we know the right hand side of (\ref{eqn:4.13}) is
negative. This is a contradiction. So, the solution of $(P_s)$ is unique.
\end{proof}

\begin{thm}\label{thm:4.2}
Suppose that $g(x)\in C^\infty([0,\,\pi])$ satisfies $(i)-(iii)$. Then, the problem $Q_{(s,\mu)}$ admits a unique solution $h(r)$. Moreover,
$h'(r)>0$ for any $r \in [s,\,\mu]$.
\end{thm}

\begin{proof}
By replacing $+\infty$ by $\mu$, and $X_s$ by $Y_s$, where
$$Y_s=\{h(r)\ |\  \pi-h(r) \in H^1([s,\mu],rdr),\  \xi \le h(r) \le \pi,\  h(s)=\xi \quad \text{and} \quad h(\mu)=\pi\},$$
in the proof of Theorem \ref{thm:4.1}, we can address the existence of solution to the problem $Q_{(s,\mu)}.$

By Lemma \ref{lem:3.3_1}, we get $h'(r)>0,\ r \in [s,\mu]$. By
Lemma \ref{lem:3.4}, we get the uniqueness of solution to the
problem $Q_{(s,\mu)}.$

\end{proof}

\section{The proof of Theorem \ref{thm:1}}
In \cite{D1, D2}, Ding has ever employed a mini-max argument to obtain the existence and
uniqueness of equivariant harmonic maps from a sphere into another sphere. However, for our present case it seems that
Ding's method is not valid. Here, we will employ the shooting target method to prove Theorem \ref{thm:1}.
For this goal, we need to characterize the behavior $h_a(r)$, the
solution of (\ref{eqn:2.1})-(\ref{eqn:2.2}) with $a>0$. Concretely, we need to
establish some lemmas on when $h_a(r)$ increases monotonically from 0 to
$\pi$ on a finite interval. For simplicity, we would like to call such increasing $h_a(r)$ as
\emph{``solution of type $(\rm I)$"} (see the following definition
\ref{def:1}). We will employ the blow-up analysis to show that $h_a$
is actually a solution of type $(\rm I)$ as $a>0$ is small enough.
As $a>0$ is large enough, we combine the blow-up analysis and the Pohozaev
identities to characterize the behaviors of $h_a$ which is completely different
from the case that $a>0$ is small.

We know that there exists $\lambda_0 >0$ such that
$\varphi_{\lambda_0}(r) = 2 \arctan [ (\lambda_0 r)^{m}]$, the
equivariant harmonic map with degree $m$, satisfying the following
initial value problem:

\begin{equation}\label{eqa:5.1}
\left\{
\begin{aligned}
&h''+\frac{1}{r}h'-\frac{{m^2 }}{{r^2 }}\sin h \cos h=0,\quad r \in (0,+\infty),\\
&h(0)=0,\quad h^{(m)}(0)=m!.
\end{aligned}\right.
\end{equation}
For the sake of convenience, we denote $\varphi_{\lambda_0}$ by
$\phi$. As $\phi$ is increasing monotonically from $0$ to $\pi$, then there exists unique $r_{\xi} \in
(0,+\infty)$ such that $\phi(r_{\xi})=\xi$.

Define $$\phi_s(r)=\phi\left(\frac{r_{\xi}}{s}r\right).$$ Then $\phi_s(s)=\xi$
and $\phi_s(r)$ satisfies
\begin{eqnarray}\label{eqn:5.2}
h''+\frac{1}{r}h'-\frac{{m^2 }}{{r^2 }}\sin h \cos h=0,\quad r \in (0,+\infty).
\end{eqnarray}

\begin{lem}\label{lem:5.1}
Suppose that $g(x)\in C^\infty([0,\,\pi])$ satisfies $(i)-(iii)$. If $\widetilde h_s$ is the solution of $(P_s)$,
then, for any $r \in (s,+\infty)$ there holds true $$\phi_s(r)<\widetilde h_s(r).$$
\end{lem}

\begin{proof}
By the Pohozaev identity (\ref{eqn:3.1}), we have that, for $r \in
[s,+\infty)$,
\begin{eqnarray}\label{eqn:5.3}
\left( {r\widetilde h_s'(r)} \right)^2 = m^2 \sin ^2 \widetilde h_s(r) - 2\int_r^{+\infty} {g(\widetilde h_s(t))\widetilde h_s'(t)t^2 dt}
\end{eqnarray}
and
\begin{eqnarray}\label{eqn:5.4}
\left( {r\phi_s'(r)} \right)^2 = m^2 \sin ^2 \phi_s(r).
\end{eqnarray}
By Lemma \ref{lem:3.3}, we have $\widetilde h_s'(r)>0$ for any $r \in
[s,+\infty)$. It is easy to see that, for $r \in [s,+\infty)$,
$$- 2\int_r^{+\infty} {g(\widetilde h_s(t))\widetilde h_s'(t)t^2 dt}>0,$$
here we also use the fact $g(x)\le 0$ for any $x \in [\xi,\pi]$.
Since $\widetilde h_s(s)=\phi_s(s)=\xi$, by comparing
(\ref{eqn:5.3}) and (\ref{eqn:5.4}) we obtain that
$$\phi_s'(s)<\widetilde h_s'(s).$$ In fact, there holds true that,
for any $r \in (s,+\infty)$,
$$\phi_s(r)<\widetilde h_s(r).$$
If the above inequality fails, we define
$$r^*=\sup \{r \in (s,+\infty)\ |\ \widetilde h_s(t)>\phi_s(t),\ t
\in (s,r)\}$$
and have
$$s<r^*<+\infty.$$
By the definition of $r^*$, we have
$$\phi_s(r)<\widetilde h_s(r),\quad\forall r \in (s,\,r^*)\quad \text{and}\quad
\phi_s(r^*)=\widetilde h_s(r^*).$$
This implies that $$\phi_s'(r^*)
\ge \widetilde h_s'(r^*).$$ However, from (\ref{eqn:5.3}) and
(\ref{eqn:5.4}) we infer that$$\phi_s'(r^*)<\widetilde h_s'(r^*),$$
as $\phi_s(r^*)=\widetilde h_s(r^*)$. A contradiction is derived.
Thus, we complete the proof of the lemma.

\end{proof}

Let $h_a$ be the solution of (\ref{eqn:2.1})-(\ref{eqn:2.2}) with $a>0$. By Corollary \ref{coro:3.2},
we know there exists $s_a \in (0,+\infty)$ such that $h_a(r)$ increases monotonically from $0$ to $\xi$ in the interval $[0,s_a]$. Then,we have the following lemma.

\begin{lem}\label{lem:5.2}
Suppose that $g(x)\in C^\infty([0,\,\pi])$ satisfies $(i)-(iii)$ and $h_a$ is a solution of (\ref{eqn:2.1})-(\ref{eqn:2.2}) with $a>0$,
increases monotonically from $0$ to $\xi$ on the interval $[0,s_a]$.
Then, there holds true $\phi_{s_a}(r)> h_a(r)$ for any $r \in (0,\, s_a)$.
\end{lem}
\begin{proof}
The proof is analogous to the proof of Lemma \ref{lem:5.1}. We omit it.
\end{proof}

\begin{lem}\label{lem:5.3}
Suppose that $g(x)\in C^\infty([0,\,\pi])$ satisfies $(i)-(iii)$. If $\widetilde h_s$ is the solution of $(P_s)$ and $h$ satisfies (\ref{eqn:2.1})
on the interval $(s,+\infty)$ with $h(s)=\xi$, then, there exists $r_1 \in (s,+\infty)$
such that $h(r)$ increases monotonically from $\xi$ to $\pi$ on the interval $[s,r_1]$
if and only if  $h'(s)>\widetilde h_s'(s)$.
\end{lem}

\begin{proof}
(1). If $h'(s)<\widetilde h_s'(s)$, there does not exist $r_1 \in (s,+\infty)$ such that $h(r)$ increases monotonically
from $\xi$ to $\pi$ on the interval $[s,r_1]$. Otherwise, $h$ will intersect $\widetilde h_s$ at two different
points in the domain $[s,+\infty) \times [\xi,\pi]$ which contradicts Lemma \ref{lem:3.4}.\\

(2). If $h'(s)=\widetilde h_s'(s)$, by the uniqueness of initial value problem, we get $h \equiv \widetilde h_s$. As
$\widetilde h_s$ increases monotonically from $\xi$ asymptotically to $\pi$ on the interval $[s,+\infty)$, we can't pick
$r_1 \in (s,+\infty)$ such that $h(r)$ increases monotonically from $\xi$ to $\pi$ in the interval $[s,\, r_1]$.\\

(3). For simplicity, let $h(r,a)$ be the solution of the following problem:
\begin{equation}
\left\{
\begin{aligned}
&(rh')' = \frac{{m^2 }}{r}\sin h \cos h  + g(h)r, \quad r \in (s,+\infty)\\
&h(s)=\xi, \quad h'(s)=a.\nn\\
\end{aligned}\right.
\end{equation}
Then, we have $h(r,\widetilde h_s'(s)) \equiv \widetilde h_s(r).$

Define $$A=\{a \in \mathbb{R} \,|\,h(r,a)\, \text{ increases monotonically from}\,  \xi \,\text{to}\, \pi \, \text{in finite interval}\}$$
By Theorem \ref{thm:4.2}, we have $A$ is a non-empty set. Moreover, by the continuous dependence of the solutions on the initial data
and Lemma \ref{lem:3.3_1}, we know $A$ is a non-empty open set. From the argument in (1) and (2), we get $$\inf \{a \,|\,a \in A\}\ge \widetilde h_s'(s)>0.$$

For any $\bar a \in \partial A$ (the boundary of $A$) and $\bar a<+\infty$, we can choose a sequence $\{a_k\}\subseteq A$ such that $$\lim\limits_{k \to \infty }a_k=\bar a.$$
Let $r_k$ be the minimal number $r \ge s$ such that $h(r,a_k)=\pi$. Then, we have $$\mathop {\lim }\limits_{k \to \infty }r_k = +\infty.$$ Otherwise, by the continuous
dependence of the solutions on the initial data, we have $\bar a \in A$ which contradicts the fact $A$ is a open set. By the continuous
dependence of the solutions on the initial data again, we have that, for any $r \in [s,+\infty)$,
$$h'(r,\bar a) \ge 0\quad\text{and} \quad \xi \le h(r,\bar a) \le \pi.$$
It deduces that $$\lim\limits_{r \to \infty }h(r,\bar a)=l>\xi.$$
By Lemma \ref{lem:3.3}, we obtain $$\xi<h(r,\bar a)<\pi,\quad r \in (s,+\infty) \quad \text{and}\quad l=\pi.$$
It means that $h(r,\bar a)$ is also a solution of problem $(P_s)$. By the uniqueness of solution to the problem $(P_s)$,
we get $\bar a=\widetilde h_s'(s).$ So, we get $A=(\widetilde h_s'(s),+\infty)$. The proof of the lemma is completed.

\end{proof}

\begin{defnm}\label{def:1}$h_a$ is called a solution of type $(\rm I)$ to (\ref{eqn:2.1})-(\ref{eqn:2.2}) if
there exists $r_a
\in (0,+\infty)$ such that $h_a$ increases monotonically from $0$ to
$\pi$ in the interval $[0,r_a]$.
\end{defnm}

Let $h_a$ be the solution of problem (\ref{eqn:2.1})-(\ref{eqn:2.2}) with $a>0$. By Corollary \ref{coro:3.2}, let $s_a$ be the minimal positive number such that $h_a(s)=\xi$. Then we have\\

(a). $s_a \to +\infty$ as $a \to 0$.\\

(b). By the Pohozaev identity
\begin{eqnarray}\label{eqn:5.5}
\left( {rh_a'(r)} \right)^2 = m^2 \sin ^2 h_a(r) + 2G(h_a(r))r^2 - 4\int_0^r {G(h_a(t))tdt},
\end{eqnarray}
we get
\begin{eqnarray}
\left( h_a'(r) \right)^2 &\le& \frac{1}{r^2}\{m^2 \sin ^2 h_a(r) + 2G(\xi)r^2 - 4\int_0^r {G(0)tdt}\}\nn\\
&\le& m^2 \frac{\sin^2 h_a(r)}{r^2}) + 2(G(\xi)-G(0)). \nn
\end{eqnarray}
On the other hand, by (\ref{eqn:5.5}) and Lemma \ref{lem:5.2}, we have
\begin{eqnarray}
\left( h_a'(s_a) \right)^2 &\ge& \frac{1}{{s_a}^2}\{m^2 \sin ^2 \xi + 2G(\xi){s_a}^2 - 4\int_0^{s_a} {G(\phi_{s_a}(t))tdt}\} \nn\\
&\ge& 2G(\xi)-2\frac{1}{r_{\xi}^2}\int_0^{r_{\xi}}{G(\phi(t))tdt}>0. \nn
\end{eqnarray}
Combining the above two inequalities with the properties of $h_a$ near $r=0$ (cf. Section 2), we have that
there exists positive constant $c_0$ and $c_1$ such that, if $a \le 1$, then
$$0<h_a'(r)\le c_1, \quad \forall r \in (0,s_a]  \quad \text{and} \quad 0<c_0 \le h_a'(s_a)\le c_1.$$

\begin{lem}\label{lem:5.5'}
Suppose that $g(x)\in C^\infty([0,\,\pi])$ satisfies $(i)-(iii)$ and $h_a$ is a solution to (\ref{eqn:2.1})-(\ref{eqn:2.2}).
Then, there exists a positive number $b>0$ such that, if $a >0$ is small enough, then, we have $$ (h_a'(s_a))^2 \ge 2G(\xi)+b.$$
\end{lem}

\begin{proof}
Let $a_i$ be any sequence such that $a_i>0$ and $a_i \to 0$. For simplicity, we denote $h_i=h_{a_i}$ and $s_i=s_{a_i} \to +\infty$.

Set $$u_i(r)=h_i(r+s_i).$$
Then, we have that, on $(-s_i,\,0]$,
$$u_i''+\frac{1}{r+s_i}u_i'-\frac{m^2}{(r+s_i)^2}\sin u_i \cos u_i=g(u_i).$$

Note that we have $u_i(0)=\xi$, $c_0 \le u_i'(0) \le c_1$,
$$0<u_i(r)\le \xi \quad \text{and} \quad 0<u_i'(r) \le c_1,\quad \forall r \in (-s_i,0].$$
By a diagonal subsequence argument, there exists a subsequence of $\{u_i\}$, still denoted by $u_i$, such that $u_i$
converges uniformly in $C^2([-R,0])$, for any given $R>0$, to some $u \in C_{loc}^2(-\infty,0])$.
The limit $u$ satisfies the following equation on $(-\infty,0]$
\begin{eqnarray}\label{limitf}
u''=g(u),
\end{eqnarray}
with $c_0 \le u'(0) \le c_1$. Moreover, for any  $r \in (-\infty,0]$
there holds true
$$0\le u(r) \le \xi=u(0)\quad\text{and}\quad 0 \le u'(r) \le c_1.$$
Hence, we conclude that there exists $0\le l \le \xi$ such that
$$\mathop {\lim }\limits_{r \to -\infty }u(r)=l.$$
Since $u''=g(u) \ge 0$, it is easy to see that there exists $\mu\in [0,\, c_1]$
such that
$$\mathop {\lim }\limits_{r \to -\infty }u'(r)=\mu.$$

On the other hand, we also have $$\int_{-\infty}^0{u'(r)dr}\le \xi.$$
Hence, we infer that $u'(r) \to 0$ as $r \to -\infty$, i.e. $\mu=0.$
It follows that $$\int_{-\infty}^0{u''(r)dr}=u'(0).$$
As $u'' \ge 0$, from the integrability of $u''$ on $(-\infty, 0]$ we conclude that
$$\mathop {\lim }\limits_{r \to -\infty }u''(r)=\mathop {\lim }\limits_{r \to -\infty }g(u(r))=g(l)=0.$$
Hence, $l=0$ or $\xi$. But, since $u$ is increasing function on $(-\infty,0]$ with $u(0)=\xi$ and $u'(0) \ge c_0>0$,
we have $l<\xi$. Hence, $l=0$.

By integrating the two sides of (\ref{limitf}) we obtain
\begin{eqnarray}\label{5.8}
u'(r)^2=2G(u(r))+C,
\end{eqnarray}
where $C=(u'(0))^2-2G(\xi)$. Let $r \to -\infty$ in (\ref{5.8}), we deduce
\begin{eqnarray}\label{5.9}
u'(0)^2=2G(\xi)-2G(0).
\end{eqnarray}

Notice that, the solution $u$ of (\ref{limitf}) with initial data $u(0)=\xi$ and $u'(0)$ is unique.
The uniqueness implies that, if we denote $u_a(r)=h_a(r+s_a)$, there holds true
\begin{center}
$u_a(r) \to u(r)$ uniformly in $C^2([-R,0])$ for any $R>0$,
\end{center}
as $a \to 0$. Then, by (\ref{5.9}), $$h_a'(s_a)^2 \to 2G(\xi)-2G(0).$$
If we first take $b=-G(0)$, then let $a>0$ be small enough, then the desired conclusions follow.
Thus we complete the proof.

\end{proof}

\begin{thm}\label{thm:5.4'}
Suppose that $g(x)\in C^\infty([0,\,\pi])$ satisfies $(i)-(iii)$. If $h_a$ is the solution of (\ref{eqn:2.1})-(\ref{eqn:2.2}) with
$a>0$, then, there exists $\epsilon>0$ such that for any $a \in
(0,\epsilon)$, $h_a$ is a solution of type $(\rm I)$.
\end{thm}

\begin{proof}
By Corollary \ref{coro:3.2}, for any $a>0$ there exists $s_a \in (0,+\infty)$
such that $h_a(r)$ increases monotonically from $0$ to $\xi$ on the interval
$[0,s_a]$ with $h(s_a)=\xi$ and $ h'(s_a)>0.$

Let $\widetilde h_{s_a}$ be the solution of the problem $(P_s)$ with
$s=s_a$. Then, by Lemma \ref{lem:5.3}, $h_a$ is a solution of type
$(\rm I)$ if and only if $h_a'(s_a)>\widetilde h_{s_a}'(s_a)$. So,
to prove the theorem, it is sufficient to prove that there exists $\epsilon>0$ such
that, for all $a \in (0,\epsilon)$,
$$h_a'(s_a)>\widetilde h_{s_a}'(s_a).$$

By Theorem \ref{thm:4.1}, we know $\widetilde h_{s_a}$ minimizes the
functional $J_{s_a}$ on space $X_{s_a}$. Let
$$\theta (r) = \left\{ \begin{array}{l}
 (\pi-\xi)(r-s_a)+\xi, \quad r \in [s_a,s_a+1], \\
 \pi, \quad  r \in (s_a+1,+\infty) \\
 \end{array} \right.$$
then, $\theta \in X_{s_a}.$
It follows that $J_{s_a}(\widetilde h_{s_a}) \le J_{s_a}(\theta).$  Hence, we have
\begin{eqnarray}\label{eqn:pie}
\int_{s_a}^{+\infty}{G(\widetilde h_{s_a})tdt} &\le& J_{s_a}(\widetilde h_{s_a}) \le J_{s_a}(\theta)\nn\\
&\le& \frac{1}{2}\int_{s_a}^{s_a+1}{[(\pi-\xi)^2+\frac{{m^2}}{{r^2}}]rdr} + \int_{s_a}^{s_a+1}{G(\xi))rdr}\nn\\
&\le& C\left\{(s_a+\frac{1}{2})+ \log \frac{s_a+1}{s_a}\right\},
\end{eqnarray}
where $C$ is a positive constant independent of $s_a$.

On the other hand, by the Pohozaev identity we have
\begin{eqnarray}
(s_a \widetilde h_{s_a}'(s_a)) ^2 = m^2 \sin ^2 \xi + 2G(\xi){s_a}^2 + 4\int_{s_a}^{+\infty} {G(\widetilde h_{s_a}(t))tdt}.
\end{eqnarray}
Combining (\ref{eqn:pie}) and the above identity we obtain
\begin{eqnarray}
(\widetilde h_{s_a}'(s_a)) ^2 \le \frac{1}{s_a^2}\left\{ m^2 \sin ^2 \xi + 2G(\xi){s_a}^2 + 4C((s_a+\frac{1}{2})+ \log \frac{s_a+1}{s_a})\right\}.
\end{eqnarray}
Hence, it follows
$$\mathop {\varlimsup }\limits_{a \to 0}\widetilde h_{s_a}'(s_a) \le 2G(\xi).$$
Combining the last inequality with Lemma \ref{lem:5.5'}, there
exists $\epsilon>0$ such that for all $a \in (0,\epsilon)$,
$$h_a'(s_a)>\widetilde h_{s_a}'(s_a).$$ Then the proof of the theorem is finished.

\end{proof}

\begin{thm}\label{thm:5.4}
Suppose that $g(x)\in C^\infty([0,\,\pi])$ satisfies $(i)-(iii)$. Let $\phi$ be the solution of the problem (\ref{eqa:5.1}).\\
$(1).$ If $$-\infty<\int_0^{+\infty}{G(\phi)rdr}\le 0,$$
then $h_a$  is a solution of type $(\rm I)$ to (\ref{eqn:2.1})-(\ref{eqn:2.2}) for all $a>0$.\\
$(2).$ If $$0<\int_0^{+\infty}{G(\phi)rdr}\le +\infty,$$
then there exists $a_0>0$ such that $h_a$ is not a solution of type $(\rm I)$
to (\ref{eqn:2.1})-(\ref{eqn:2.2}) for $a>a_0$.
\end{thm}

\begin{proof}
By Corollary \ref{coro:3.2}, for $a>0$, there exists $s_a \in (0,+\infty)$
such that $h_a(r)$ increases monotonically from $0$ to $\xi$ in the interval
$[0,\,s_a]$ with $h(s_a)=\xi$ and $h'(s_a)>0$.
Let $\widetilde h_{s_a}$ be the solution of problem $(P_s)$ with $s=s_a$.
Then, by the Pohozaev identity, we have
\begin{eqnarray}
&&(s_ah_a'(s_a)) ^2 = m^2 \sin ^2 \xi + 2G(\xi){s_a}^2 - 4\int_0^{s_a} {G(h_a(t))tdt},\label{eqn:5.12}\\
&&(s_a \widetilde h_{s_a}'(s_a)) ^2 = m^2 \sin ^2 \xi + 2G(\xi){s_a}^2 + 4\int_{s_a}^{+\infty} {G(\widetilde h_{s_a}(t))tdt}.\label{eqn:5.13}
\end{eqnarray}

Now we discuss the case $(1)$.  By Lemma \ref{lem:5.3}, $h_a$ is a solution of
type $(\rm I)$ if and only if $$h_a'(s_a)>\widetilde
h_{s_a}'(s_a).$$ Comparing (\ref{eqn:5.12}) and (\ref{eqn:5.13}), it
suffices to prove that, for $a>0$, the following inequality is true
\begin{eqnarray}\label{eqn:5.14}
\int_0^{s_a} {G(h_a(t))tdt}+\int_{s_a}^{+\infty} {G(\widetilde h_{s_a}(t))tdt}<0.
\end{eqnarray}
Since $G(x)$ is increasing on the interval $[0,\,\xi]$ and decreasing
on the interval $[\xi,\,\pi]$, by Lemma \ref{lem:5.2} and Lemma
\ref{lem:5.1}, we derive that, as $a>0$,
\begin{eqnarray}
&&\int_{0}^{s_a} {G(h_a(t))tdt} < \int_{0}^{s_a} {G(\phi_{s_a}(t))tdt}= \frac{s_a^2}{r_{\xi}^2}\int_{0}^{r_{\xi}}{G(\phi(t))tdt},\label{eqn:5.15}\\
&&\int_{s_a}^{+\infty} {G(\widetilde h_{s_a}(t))tdt} < \int_{s_a}^{+\infty} {G(\phi_{s_a}(t))tdt}= \frac{s_a^2}{r_{\xi}^2}\int_{r_{\xi}}^{+\infty}{G(\phi(t))tdt}.\label{eqn:5.16}
\end{eqnarray}
Combining (\ref{eqn:5.15}) and (\ref{eqn:5.16}), we get that, for
$a>0$, there holds
$$\int_0^{s_a} {G(h_a(t))tdt}+\int_{s_a}^{+\infty} {G(\widetilde h_{s_a}(t))tdt}<\frac{s_a^2}{r_{\xi}^2}\int_{0}^{+\infty}{G(\phi(t))tdt}.$$
So, when
$$-\infty<\int_0^{+\infty}{G(\phi)rdr}\le 0,$$
then (\ref{eqn:5.14}) follows and the conclusion stated in (1) is true.\\

We turn to the discussion of the case $(2)$. If the conclusion stated in $(2)$ fails, then there exists a sequence $a_i \to +\infty$
such that $h_{a_i}$ is a solution of type $(\rm I)$ to (\ref{eqn:2.1})-(\ref{eqn:2.2}).

Set $$\hbar_i({a_i}^{\frac{1}{m}}r)\equiv h_{a_i}(r).$$
Then $\hbar_i(r)$ is the solution of the following problem:
\begin{equation}\label{eqa:5.8}
\left\{
\begin{aligned}
&\hbar_i''(r) + \frac{1}{r} \hbar_i'(r) - \frac{{m^2 }}{{r^2 }} \sin \hbar_i(r) \cos \hbar_i(r)-{a_i}^{-\frac{2}{m}}g(\hbar_i (r))=0.\\
&\hbar_i(0)=0,\quad \hbar_i^{(m)}(0)=m!
\end{aligned}\right.
\end{equation}
Let $s_i$ be the minimal positive number $s>0$ such that $h_i(s)=\pi$.
Comparing the problem (\ref{eqa:5.8}) with the problem (\ref{eqa:5.1}), we conclude that, for any $R>0$,
$\hbar_i$ converges to $\phi$ uniformly in $C^1[0,R]$ as $a_i \to +\infty$.
Notice that, $\phi$ increases monotonically from $0$ asymptotically to $\pi$,
then, we have $$\mathop {\lim }\limits_{i \to +\infty }s_i=+\infty.$$

By the Pohozaev identity, we have
$$(s_i\hbar_i'(s_i))^2=-4{a_i}^{-\frac{2}{m}}\int_0^{s_i}G(\hbar_i)rdr.$$
Since $\hbar_i'(s_i)>0$, we get
\begin{eqnarray}\label{5.14}
\int_0^{s_i}G(\hbar_i)rdr<0.
\end{eqnarray}

On the other hand, since $$0<\int_0^{+\infty}{G(\phi)rdr}\le +\infty,$$
we can pick $R_0>0$ such that $\phi(R_0)>\xi$ and
\begin{eqnarray}\label{5.15}
\int_0^{R_0}{G(\phi)rdr}>0.
\end{eqnarray}
Hence, $$\int_0^{R_0}{G(\phi)rdr}=\lim\limits_{i \to +\infty }
\int_0^{R_0}{G(\hbar_i)rdr} \le \varliminf\limits_{i \to \infty }\int_0^{s_i}{G(\hbar_i)rdr} \le 0,$$
which contradicts (\ref{5.15}). So, if $$0<\int_0^{+\infty}{G(\phi)rdr}\le +\infty,$$
there always exists $a_0>0$ such that, for $a>a_0$, $h_a$ is not a solution of type $(\rm I)$ to
(\ref{eqn:2.1})-(\ref{eqn:2.2}).

\end{proof}

Now, we are in the position to show Theorem \ref{thm:1}.\\

\noindent{\bf\textit{Proof of Theorem \ref{thm:1}}}.
Let $\phi$ be the solution of problem (\ref{eqa:5.1}).
From the Remark \ref{rem:1}, we can replace $\varphi_1$ by $\phi$.
Since $\phi$ increases from $0$ asymptotically to $\pi$ and $G(x)\ge 0$
for $x \in [\xi,\,\pi]$, we have $$-\infty<\int_0^{+\infty}{G(\phi)rdr}\le +\infty.$$

Define
$$A=\{a>0 \  | \  h_a \text{ is a solution of type $(\rm I)$ to (\ref{eqn:2.1})-(\ref{eqn:2.2})}\}.$$
By the continuous dependence of the solutions on the initial data (cf. Theorem \ref{thm:2.2}), Corollary \ref{coro:3.2} and Lemma \ref{lem:3.3_1},
we know $A$ is a open set. By Theorem \ref{thm:5.4'}, we derive that $A$ is a non-empty open set.

We need only to consider the following two cases:

\noindent Case $(1).$ If $$-\infty<\int_0^{+\infty}{G(\phi)rdr}\le 0,$$
Theorem \ref{thm:5.4} tells us that $A=(0,+\infty)$. It means that
all solutions of (\ref{eqn:2.1})-(\ref{eqn:2.2}) with $a>0$ increase from $0$ to $\pi$ on finite interval.
So the problem (\ref{(3)})-(\ref{(4)}) with $0< h(r) < \pi$ on $(0, \infty)$ doesn't admits any solution.

\noindent Case $(2).$ If
$$0<\int_0^{+\infty}{G(\phi)rdr}\le +\infty,$$
by Theorem \ref{thm:5.4} we have $$a^*=\sup \{a \in A\} <+\infty.$$
We claim that $h_{a^*}$ is a solution of (\ref{(3)})-(\ref{(4)}) with $0< h(r) < \pi$ on $(0, \infty)$.

Now, we always assume $a \in A$. Let $r_a \in (0,+\infty)$ such that $h_a$ increases monotonically
from $0$ to $\pi$ on the interval $[0,\, r_a]$. Then, we have $\lim\limits_{a \to a^* }r_a=+\infty$.
Otherwise, there exists a sequence $a_k \to a^*$ such that $$\lim\limits_{a_k \to a^* }r_{a_k}=r^*$$
where $r^* \in (0,+\infty)$. By Theorem \ref{thm:2.2}, we have that $h_{a^*}$ is also a
solution of  type $(\rm I)$, which contradicts the definition of $a^*$.

As $\lim\limits_{a \to a^* }r_a=+\infty$, Theorem \ref{thm:2.2} tells us that $$h_{a^*}'(r) \ge 0\quad  \text{and}\quad  0\le h_{a^*}(r)\le \pi,\quad r \in (0,+\infty).$$
Then, we have $$\lim\limits_{r \to +\infty }h_{a^*}(r)=l.$$
By Corollary \ref{coro:3.2} and the fact $\pi$ is a trivial solution of equation (\ref{eqn:2.1}),
we have that $$0<h_{a^*}(r)<\pi\quad\text{for any}\quad r \in (0,+\infty) \quad\quad \text{and}\quad \quad\xi<l\le \pi.$$
By Lemma \ref{lem:3.3}, we have that $l=\pi$ and $h_{a^*}(r)$ converge to $\pi$ exponentially as $r \to +\infty$.
Hence $h_{a^*}(r)$ is a solution of (\ref{(3)})-(\ref{(4)}) with $0< h(r) < \pi$ on $(0, \infty)$. Thus, we complete the proof of the theorem.

\section {Applications to the equations of Landau-Lifshitz type}

In this section, we will generalize and improve the results due to Gustafson and Shatah in \cite{G_S} as an application of Theorem \ref{thm:1}.
First we recall the Landau-Lifshitz equation
\begin{eqnarray*}
\partial _t u = u \times \Delta u,
\end{eqnarray*}
where $u: M\times\mathbb{R}\to S^2$ and $``\times"$ denotes the cross product in $\mathbb{R}^3$ (see \cite{LD}). We would like to consider
the Schr\"{o}dinger flows for maps $u: M\times\mathbb{R}\to S^2$ corresponding to the functional
$$\widetilde F(u)=\int_M|\nabla u|^2dM + \int_M \widetilde H(u)dM$$
in the sense of \cite{DW}. The equation of flows can be expressed as
\begin{eqnarray}\label{eqn:dw}
u_t=u\times(\Delta u - \nabla \widetilde H(u)).
\end{eqnarray}
The stationary solutions of this equation satisfy
$$\Delta u + |\nabla u|^2u = \nabla \widetilde H(u).$$
This is just the elliptic system of harmonic maps with potential $\widetilde H(u)$ from $M$ into $S^2$.
When $M$ is a compact Riemann surface of some symmetry and $\widetilde H\equiv 0$, Ding and Yin have studied the existence of special periodic solutions to
the Landau-Lifshitz equation. For details we refer to \cite{DY}. Such a special class of periodic solutions can also be
called as ``geometric solitons" in \cite{SW, SSW}. For the case $M\equiv\mathbb{R}^2$ and
$\widetilde H(u)\equiv \widetilde G(d(u))$, where $d(u)$ denotes the geodesic distance from $u\in S^2$ to the north pole
$P=(0,0, 1)$, we would like to consider the equivariant solutions to (\ref{eqn:dw}) written by
\begin{eqnarray}\label{eqn:6.2}
u(x,t)=(\sin h(r) \cos(m\theta+wt),\sin h(r) \sin(m\theta+wt),\cos h(r)),
\end{eqnarray}
where $(r,\,\theta)$ is the polar coordinates on $\mathbb{R}^2$ and $m \in \mathbb{Z}\backslash \{ 0\}$. Substituting (\ref{eqn:6.2}) into (\ref{eqn:dw}), we obtain
\begin{eqnarray}\label{eqn:6.3}
h''+\frac{1}{r}h'-\frac{m^2}{r^2}\sin h \cos h =\widetilde g(h)+\omega \sin h
\end{eqnarray}
where $\widetilde g(x)=\widetilde G'(x)$.

Set $$g(h)=\widetilde g(h)+\omega \sin h,$$
then equation (\ref{eqn:6.3}) is of the same form as (\ref{(3)}).

In particular, the following equation of Landau-Lifshitz type is of strong physical background
\begin{eqnarray}\label{eqn:6.1}
\partial _t u = u \times (\Delta u + \lambda u_3 \widehat k),
\end{eqnarray}
where $$u(x,t):\mathbb{R}^2 \times \mathbb{R} \to \mathbb{S}^2,\quad
\widehat k=(0,0,1),\quad  \lambda>0. $$ For more details we refer to
\cite{G_S} and the references therein. The equation (\ref{eqn:6.1})
is of a Hamiltonian structure and its Hamiltonian energy functional
is:
$$E=E_e+\lambda E_a.$$
Here, the exchange energy $E_e$ and the anisotropy energy $E_a$ are defined respectively by
$$E_e=\frac{1}{2}\int\limits_{\mathbb{R}^2 } {\left| {\nabla u} \right|^2 }\quad\text{and}\quad
E_a=\frac{1}{2}\int\limits_{\mathbb{R}^2 }{1-u_3^2}.$$
It is worthy to point out that, when $\lambda \ne 0 $, the static solutions with finite energy to (\ref{eqn:6.1}) are ruled out
by Pohozaev identity.

Gustafson and Shatah in \cite{G_S} considered the existence of such equivariant solutions to (\ref{eqn:6.1}) as (\ref{eqn:6.2}).
In the present case, the corresponding equation reads
\begin{eqnarray}\label{eqn:6.4}
h''+\frac{1}{r}h'-\frac{m^2}{r^2}\sin h \cos h - g(h)=0,
\end{eqnarray}
where $$g(h)=\left(\omega +\lambda \cos h\right)\sin h.$$
Note (\ref{eqn:6.4}) is a special case of the equation (\ref{(3)}).
By variational methods, Gustafson and Shatah in \cite{G_S} obtained the following results.

\begin{prop}
For any $m \in \mathbb{Z}\backslash \{ 0\}$  there exists a positive number $\omega_0$
with $0<\omega_0 \le \frac{1}{|m|}$ such that, if $0<\omega <\omega_0\lambda$,
the equation (\ref{eqn:6.1}) with $\lambda>0$ admits a solution of form (\ref{eqn:6.2}) with
$h(r)$ satisfying $0< h(r) < \pi$ on $(0, \infty)$.
\end{prop}

A natural problem is whether or not $\omega_0$ can be accurately
determined. On the other hand, One also wants to know whether or not
the equation (\ref{eqn:6.1}) admits a solution of form (\ref{eqn:6.2})
as $\omega \notin(0,\,\omega_0\lambda)$. In other words, when
$g(h)=\left(\omega+\lambda \cos h\right)\sin h$, where $\omega \notin(0,\,\omega_0\lambda)$, is the problem (\ref{(3)})-(\ref{(4)})
solvable? As a direct application of Theorem \ref{thm:1}, we can completely answer the above problems.
More precisely, as a direct corollary of Theorem \ref{thm:1} we have the following more general results:

\begin{thm}\label{thm:6.2}
Let $(M, g)$ be an Euclidean space $\mathbb{R}^2$. Assume that the potential
function $\widetilde H(u)=\widetilde G(d(u)): S^2\to \mathbb{R}$ with
$$\widetilde G'(x) = g(x)-\omega \sin x,$$ where $g(\,\cdot\,)$ satisfies
$(i)-(iii)$. Then the equation (\ref{eqn:dw}) admits a solution of
form (\ref{eqn:6.2}) with $h(r)$ satisfying $0< h(r)<\pi$ on $(0,
\infty)$ if and only if $G$, which is defined by $G(x)=-\int_x^\pi
g(t)dt$, satisfies
$$0< \int_0^\infty G(\varphi_1(r)) \, r dr \leq \infty.$$
\end{thm}

Especially, for the equation (\ref{eqn:6.1}) with $\lambda>0$ we
have the following

\begin{thm}\label{thm:6.3}
For any $m \in \mathbb{Z}\backslash \{ 0\}$ the equation (\ref{eqn:6.1}) with $\lambda>0$
admits a solution of form (\ref{eqn:6.2}) with $h(r)$ satisfying $0< h(r)<\pi$ on $(0, \infty)$
if and only if $$0<\omega<\frac{\lambda}{|m|}.$$
\end{thm}

\begin{proof} In order to prove the theorem, we
only need to show the solvability of (\ref{(3)})-(\ref{(4)}) with
$$g(h)=\left(\omega+\lambda \cos h\right)\sin h.$$ If the problem
(\ref{(3)})-(\ref{(4)}) with $g(h)$ as above admits a solution
$h(r)$, by Pohozaev identity, we get
$$\int_0^{+\infty}{G(h(r))rdr}=0,$$ where
$$G(h)=\frac{\lambda}{2}\sin^2 h -\omega(1+\cos h).$$

Obviously, when $\frac{\omega}{\lambda} \le 0$, there holds true $G(x)\ge 0$ for any $x \in [0,\pi]$.
On the other hand, when $\frac{\omega}{\lambda} \ge 1$ we have $G(x)\le 0$ for any $x \in [0,\pi]$.
No matter which case happens, for any function $h(r)$ with $0 \le h(r) \le \pi$ on $(0, \infty)$ there holds true
$$\int_0^{+\infty}{G(h(r))rdr}=0,$$  if and only if
$$G(h)\equiv 0.$$
This implies $h \equiv \pi$ or $0$, since $\pi$ is an only zero
point of $G(x)$ on the interval $[0,\pi]$ as
$\frac{\omega}{\lambda}<0$ or $\frac{\omega}{\lambda}\ge 1$ while
there exist only two zero points of $G(x)$ on the interval $[0,\pi]$
in the case $\frac{\omega}{\lambda}=0$, i.e. $0$ and $\pi$. Hence
the problem (\ref{(3)})-(\ref{(4)}) doesn't admit a solution with
$0< h(r) < \pi$ on $(0, \infty)$ if $\omega$ satisfies
$\frac{\omega}{\lambda} \le 0$ or $\frac{\omega}{\lambda} \ge 1$.

Next, we need only to consider the solvability of
(\ref{(3)})-(\ref{(4)}) in the case $0<\frac{\omega}{\lambda}<1$.
For this case, it's easy to check $g(x)$ satisfies the conditions
$(i)-(iii)$.

By a direct calculation, we have $$\int_0^{+\infty}{G(\varphi_1)rdr}= \left\{ \begin{array}{l}
 +\infty,  \quad |m|=1, \\
 \left(\frac{\lambda}{|m|}-\omega\right)C_0,  \quad |m|\ge 2. \\
 \end{array} \right.$$
Here $$C_0=\frac{1}{|m|}{\rm
B}\left(\frac{1}{|m|},1-\frac{1}{|m|}\right)$$ and ${\rm B}(x,y)$ is
the Beta function and $\varphi_1=2 \arctan (r^{m}).$ So, the
conclusions of Theorem \ref{thm:6.3} follow directly from Theorem
\ref{thm:1}.
\end{proof}

{\bf Acknowledge} The author would like to thank his supervisors
Professor Weiyue Ding and Professor Youde Wang for their
encouragement and inspiring advices. In particular, without the
beneficial discussions with Professor Ding and his help, this paper
would not have been completed.

\noindent Ruiqi Jiang\\
Academy of Mathematics and Systems Science\\
Chinese Academy of Sciences,\\
Beijing 100190, P.R. China.\\
Email: jiangruiqi@amss.ac.cn

\end{document}